\documentclass[11pt]{article}
\usepackage{amsmath, amsthm, amssymb}
\usepackage{fullpage}

\title{On the monotonicity of the correction term\\
in Ramanujan's factorial approximation}

\author{Mark B. Villarino\\
Daniel Campos Salas\\
Javier Carvajal Rojas\\[6pt]
Escuela de Matem\'atica, Universidad de Costa Rica\\
San Jos\'e 2060, Costa Rica}

\newtheorem{teo}{Theorem}
\newtheorem{propn}[teo]{Proposition}
\newtheorem{corl}[teo]{Corollary}

\newcommand{\bN}{\mathbb{N}}  
\newcommand{\dsp}{\displaystyle} 
\newcommand{\Matha}{\textit{Mathematica}$\textsuperscript{\textregistered}$}
\newcommand{\word}[1]{\quad\text{#1}\quad} 

\begin{document}

\maketitle


\begin{abstract}
We present two new proofs of the monotonicity of the correction term
$\theta_n$ in \textsc{Ramanujan's} refinement of \textsc{Stirling's}
formula.
\end{abstract}


\section{Introduction} 

\textsc{Stirling's} approximation $n! \approx \sqrt{2\pi n}\,(n/e)^n$
is one of the most important results in mathematics. The Indian
mathematician \textsc{Srinivasa Ramanujan} \cite{R} proposed the
following refinement,
\begin{equation}
\boxed{
n! = \sqrt{\pi} \biggl( \dfrac{n}{e} \biggr)^n
\biggl( 8n^3 + 4n^2 + n + \dfrac{\theta_n}{30} \biggr)^{1/6}
}
\label{eq:uno} 
\end{equation}
where $\frac{3}{10} < \theta_n < 1$ and $\theta_n \to 1$ if
$n \to \infty$. In 2001, \textsc{Karatsuba}~\cite{K} proved
Ramanujan's approximation and gave a very complicated proof of the
monotonicity of the correction term, $\theta_n$, for all real
$n \geq 1$. In 2006, \textsc{Hirschhorn}~\cite{Hir} proved a more
exact version of Ramanujan's inequalities for $\theta_n$, but did not
even mention the monotonicity of~$\theta_n$. In this paper we present
two simple proofs of the monotonicity of the sequence
$(\theta_n)_{n\in\bN}$. The first proof follows directly from the
result due to Hirschhorn~\cite{Hir}. The second uses some methods
from Hirschhorn's paper; however, instead of using his double
inequality for $\theta_n$, it only uses a new and simpler lower bound.


\begin{teo} 
\label{th:thetan}
The sequence $(\theta_n)_{n\in\bN}$ is strictly increasing.
\end{teo}


\section{First Proof} 

Hirschhorn~\cite{Hir} establishes that $\theta_n$ satisfies the
inequalities
\[
1 - \frac{11}{8n} + \frac{5}{8n^2} < \theta_n
< 1 - \frac{11}{8n} + \frac{11}{8n^2}.
\]
We will now prove that $\theta_n$ increases with $n$. Define
\begin{align*}
& \alpha_n := 1 - \frac{11}{8n} + \frac{5}{8n^2}
\\[\jot]
& \beta_n := 1 - \frac{11}{8n} + \frac{11}{8n^2}. 
\end{align*}

\begin{propn} 
\label{pr:escalera}
For $n\geq 3$ the inequality $\beta_n \leq \alpha_{n+1}$ is valid.
\end{propn}

\begin{proof}
Note that
\begin{alignat*}{2}
\beta_n \leq \alpha_{n+1}
&\iff \quad & 1 - \frac{11}{8n} + \frac{11}{8n^2} 
&\leq 1 - \frac{11}{8(n+1)} + \frac{5}{8(n+1)^2}
\\
&\iff & -\frac{11}{n} + \frac{11}{n^2} 
&\leq -\frac{11}{n+1} + \frac{5}{(n+1)^2}
\\
&\iff & -\frac{11}{n(n+1)} + \frac{11}{n^2} &\leq \frac{5}{(n+1)^2}
\\
&\iff & \frac{11}{n^2(n+1)} &\leq \frac{5}{(n+1)^2}
\\
&\iff & 11n + 11 &\leq  5n^2\\
&\iff & 0 &\leq  5n^2 - 11n - 11;
\end{alignat*}
and the last inequality is true for $n\geq 3$; thus, we obtain the
result.
\end{proof}


\begin{proof}[Proof of Theorem~\ref{th:thetan}]
Note that the last result implies that 
$\theta_n < \beta_n \leq \alpha_{n+1} < \theta_{n+1}$, and
consequently $\theta_n < \theta_{n+1}$. So it is enough examine
independently the cases $n = 1$ and $n = 2$; from \eqref{eq:uno} it
follows easily that
\[
\theta_n = 30 \Biggl( \biggl( \frac{n!}{\sqrt{\pi}(n/e)^n} \biggr)^6
- 8n^3 - 4n^2 - n \Biggr).
\]
Hence, evaluating at $n = 1$, $n = 2$ and $n = 3$ directly, we obtain
\begin{align*}
\theta_1 &= 0{,}3359\dots  \\
\theta_2 &= 0{,}5117\dots  \\
\theta_3 &= 0{,}6305\dots
\end{align*}
Therefore, we can conclude that $\theta_1 < \theta_2 < \theta_3$, and
finally prove the theorem.
\end{proof}

\section{Second Proof} 
The following well-known inequalities will be used in the proof. 
\begin{align}
\ln(1 + x)
&\leq x - \frac{x^2}{2} + \frac{x^3}{3} - \frac{x^4}{4} + \frac{x^5}{5},
\label{eq:dos} 
\\
\ln(1 + x) &\leq x - \frac{x^2}{2} +\cdots+ \frac{x^7}{7},
\label{eq:tres} 
\\
\ln(1 + x) &\geq x - \frac{x^2}{2} +\cdots+ \frac{x^7}{7} - \frac{x^8}{8},
\label{eq:cuatro} 
\\
e^x &\geq 1 + x + \frac{x^2}{2} + \frac{x^3}{3} + \frac{x^4}{4!}.
\label{eq:cinco} 
\end{align}
The logarithmic inequalities are valid for $-1 < x \leq 1$ while the
exponential inequality is valid for all real~$x$.

Let $a_n := \dfrac{n!}{\sqrt{n}\,(n/e)^n}$. We first complete the
proof of following inequality proposed by Hirschhorn in~\cite{Hir}.

\begin{propn} 
\label{pr:log-estimado}
Utilizing the previous notation, for all $n \in \bN$ the following
inequality holds:
\begin{equation}
\ln \biggl( \frac{a_n}{a_{n+1}} \biggr)
> \frac{1}{12} \biggl( \frac{1}{n} - \frac{1}{n + 1} \biggr)
- \frac{1}{360} \biggl( \frac{1}{n^3} - \frac{1}{(n + 1)^3} \biggr).
\label{eq:seis} 
\end{equation}
\end{propn}

\begin{proof}
First note that
$\ln(a_n/a_{n+1}) = (n + \frac{1}{2}) \ln(1 + \frac{1}{n}) - 1$.
Taking $u := \frac{1}{n}$, it follows that the above inequality is
equivalent to
\[
\ln(1 + u)
> \frac{360u + 1080u^2 + 1110u^3 + 420u^4 + 27u^5 - 3u^6 - u^7}
       {180 (1 + u)^3 (2 + u)},
\]
or to
\[
180 (1 + u)^3 (2 + u) \ln(1 + u)
> 360u + 1080u^2 + 1110u^3 + 420u^4 + 27u^5 - 3u^6 - u^7.
\]

Using~\eqref{eq:cuatro} and~\Matha{} we obtain that
\begin{align*}
180 (1 + u)^3 &(2 + u) \ln(1 + u)
- (360u + 1080u^2 + 1110u^3 + 420u^4 + 27u^5 - 3u^6 - u^7)
\\
&\geq \frac{10u^7}{7} - \frac{561u^9}{14} - \frac{1455u^{10}}{14}
- \frac{1215u^{11}}{14} - \frac{45u^{12}}{2}
\\
&= \frac{u^7}{14} (20 - 561u^2 - 1455u^3 - 1215u^4 - 315u^5).
\end{align*}

The last term in parentheses is decreasing with respect to~$u$, and
direct computation for $u = \frac{1}{7}$ shows that its value is
about~$3.78$. Therefore, the original inequality is true for
$n \geq 7$. The cases $n = 1$ through $n = 6$ follow by direct
computation and this concludes the proof.
\end{proof}


{}From~\eqref{eq:seis} and the fact that
$a_\infty = \lim_{n\to\infty} a_n = \sqrt{2\pi}$, we deduce the
following corollary.

\begin{corl} 
\label{cr:an6-estimado}
If $n$ is a positive integer, then the following inequalities hold:
\[
a_n \geq a_\infty \exp\biggl( \frac{1}{12n} - \frac{1}{360n^3} \biggr)
= \sqrt{2\pi} \exp\biggl( \frac{1}{12n} - \frac{1}{360n^3} \biggr);
\]
and
\begin{equation}
a_n^6 \geq 8\pi^3 \exp\biggl( \frac{1}{2n} - \frac{1}{60n^3} \biggr).
\label{eq:siete} 
\end{equation}
\end{corl}

Using this result, we now prove the following proposition.

\begin{propn} 
\label{pr:raman-stirling}
If $n$ is a positive integer, then the following inequality is valid:
\begin{equation}
\biggl( \frac{n!}{\sqrt{\pi}(n/e)^n} \biggr)^6 - 8n^3 - 4n^2 - n
\geq \frac{1 - \frac{3}{2n}}{30}.
\label{eq:ocho} 
\end{equation}
\end{propn}

\begin{proof}
{}From~\eqref{eq:siete}, it is enough to prove that
\[
8n^3 \exp\biggl( \frac{1}{2n} - \frac{1}{60n^3} \biggr)
- 8n^3 - 4n^2 - n - \frac{1}{30} + \frac{1}{20n} \geq 0.
\]
Let $u := \frac{1}{n}$. Using \eqref{eq:cinco} and~\Matha{} we obtain
\begin{align*}
\exp \biggl( \frac{u}{2} - \frac{u^3}{60} \biggr)
&\geq 1 + \frac{u}{2} + \frac{u^2}{8} + \frac{u^3}{240}
- \frac{11u^4}{1920} - \frac{u^5}{480} - \frac{u^6}{4800}
+ \frac{u^7}{14400}
\\
&\qquad + \frac{u^8}{57600} - \frac{u^9}{1296000}
- \frac{u^{10}}{2592000} + \frac{u^{12}}{311040000}
\\
&\geq 1 + \frac{u}{2} + \frac{u^2}{8} + \frac{u^3}{240}
- \frac{11u^4}{1920} - \frac{u^5}{480} - \frac{u^6}{4800},
\end{align*}
since
\[
\frac{u^7}{14400} - \frac{u^9}{1296000} > 0   \word{and}
\frac{u^8}{57600} - \frac{u^10}{2592000} > 0  \word{for}  u \leq 1.
\]
Considering the last inequality multiplied by $8/u^3$, the expression
$\dsp \frac{8}{u^3} + \frac{4}{u^2} + \frac{1}{u} + \frac{1}{30}$ is
canceled, so we obtain that
\begin{align*}
\frac{8}{u^3} \exp\biggl( \frac{u}{2} - \frac{u^3}{60} \biggr)
- \frac{8}{u^3} - \frac{4}{u^2} - \frac{1}{u} - \frac{1}{30}
+ \frac{u}{20}
&\geq \frac{u}{240} - \frac{u^2}{60} - \frac{u^3}{600}
\\
&= \frac{u}{1200} (5 - 20u - 2u^2).
\end{align*}
The last term in brackets is decreasing with respect to~$u$, and by a
direct computation we see that its value for $u = \frac{1}{5}$ is
about $0.92$. Hence the original inequality is true for $n \geq 5$.
The cases $n = 1$ through $n = 4$ follow by direct computation, thus
we conclude the proof.
\end{proof}

Note that previous result establishes that
\[
\boxed{
\theta_n \geq 1 - \frac{3}{2n}
}
\]
which is a weaker bound than that obtained by Hirschhorn~\cite{Hir}.
Nevertheless it, alone, suffices to fully prove the monotonicity of
$\theta_n$, and \emph{this is the novelty in our second proof}. Before
presenting the main result, we introduce an inequality based on a
powers series calculated using~\Matha.


\begin{propn} 
\label{pr:cota-inferior}
If $n$ is a positive integer, $n > 1$, then the following inequality 
holds:
\begin{equation}
\frac{e(n - 1)^{n-1}}{n^{n-1}}
\geq 1 + \frac{1}{2n} + \frac{7}{24n^2} + \frac{3}{16n^3}
+ \frac{743}{5760n^4} + \frac{215}{2304n^5}.
\label{eq:nueve} 
\end{equation}
\end{propn}

\begin{proof}
Let $u := \frac{1}{n}$. First note that the inequality is equivalent
to
\[
1 + \frac{1 - u}{u} \ln(1 - u) 
\geq \ln\biggl( 1 + \frac{u}{2} + \frac{7u^2}{24} + \frac{3u^3}{16}
+ \frac{743u^4}{5760} + \frac{215u^5}{2304} \biggr),
\]
or the following inequality:
\[
1 \geq \frac{1 - u}{u} \ln\biggl( 1 + \frac{u}{1 - u} \biggr)
+ \ln\biggl( 1 + \frac{u}{2} + \frac{7u^2}{24} + \frac{3u^3}{16}
+ \frac{743u^4}{5760} + \frac{215u^5}{2304} \biggr).
\]

The inequalities \eqref{eq:tres} and~\eqref{eq:dos} for the first and
second term respectively show, using \Matha, that the right hand side
is an expression of the form $\dfrac{u^6 P(u)}{K(1 - u)^6}$, where $P$
is a polynomial whose constant term is positive. Let $Q$ be the
polynomial that consists of the constant term of~$P$ and all the terms
of~$P$ with negative coefficients. It is clear that $P(u) \geq Q(u)$
and that $Q$ is a decreasing polynomial on the real positive numbers.
{}From direct computation we obtain $Q(\frac{1}{8}) \doteq 0.00036$.
This implies that the original inequality holds at least for
$n \geq 8$. The cases $n = 1$ through $n = 7$ follow by direct
computation, and this completes the proof.
\end{proof}

The previous result allows us to give another proof of
Theorem~\ref{th:thetan}.


\begin{proof}[Second proof of Theorem~\ref{th:thetan}]
The difference
\begin{align*}
\theta_n - \theta_{n-1}
&= 30 \Biggl( \biggl( \frac{n!}{\sqrt{\pi}(n/e)^n} \biggr)^6
- \biggl( \frac{(n-1)!}{\sqrt{\pi}((n - 1)/e)^{n-1}} \biggr)^6
- (8(3n^2 - 3n + 1) + 4(2n - 1) + 1) \Biggr)
\\
&= 30 \Biggl(
\biggl( \frac{(n-1)!}{\sqrt{\pi}((n - 1)/e)^{n-1}} \biggr)^6
\biggl( \biggl( \frac{e(n - 1)^{n-1}}{n^{n-1}} \biggr)^6 - 1 \biggr)
- 24n^2 + 16n - 5 \Biggr).
\end{align*}

The inequalities \eqref{eq:ocho} and~\eqref{eq:nueve} show,
using~\Matha, that the right hand side is an expression of the form
$\dfrac{P(n)}{Kn^{30}(n - 1)}$, where $P$ is a polynomial whose
leading coefficient is positive. Let $Q$ be the polynomial that
consists of the leading coefficient of~$P$ and all the terms of~$P$
with negative coefficients. Then $Q(n) = n^{29} R(\frac{1}{n})$ where
$R$ is a polynomial increasing with respect to~$n$, and whose constant
term is positive. In addition, a direct computation shows that
$R(\frac{1}{106})/K \doteq 0.00023$, and so the result is valid for
$n \geq 106$. The remaining cases follow by direct computation, thus
completing the proof of our assertion.
\end{proof}

\section*{Acknowledgements}

We would like to thank our friend Joseph C. V\'arilly, who read a
draft of the article and offered helpful suggestions.



\end{document}